\documentclass[11pt]{article}

\usepackage[applemac]{inputenc} 

\usepackage{amsthm,amssymb,amsbsy,amsmath,amsfonts,amssymb,amscd,mathrsfs}
\usepackage[normalem]{ulem}

\long\def\symbolfootnote[#1]#2{\begingroup%
\def\thefootnote{\fnsymbol{footnote}}\footnote[#1]{#2}\endgroup} 

\newcommand{\ffoot}[1]{}

\usepackage{mathtools} 
\mathtoolsset{showonlyrefs,showmanualtags} 

\usepackage{graphicx}
\DeclareGraphicsExtensions{.jpg}
\usepackage{color}
\usepackage{fullpage}
\usepackage[hypertex]{hyperref}
\newtheorem{theorem}{Theorem}
\newtheorem{corollary}[theorem]{Corollary}
\newtheorem{lemma}[theorem]{Lemma}
\newtheorem{proposition}[theorem]{Proposition}

\theoremstyle{definition} 
\newtheorem{definition}[theorem]{Definition}
\newtheorem{example}[theorem]{Example}
\theoremstyle{remark}
\newtheorem{remark}[theorem]{Remark}

\newcommand{\bt}{\begin{theorem}}
\newcommand{\et}{\end{theorem}}
\newcommand{\bl}{\begin{lemma}}
\newcommand{\el}{\end{lemma}}
\newcommand{\bp}{\begin{proposition}}
\newcommand{\ep}{\end{proposition}}
\newcommand{\bc}{\begin{corollary}}
\newcommand{\ec}{\end{corollary}}
\newcommand{\bdeff}{\begin{definition}}
\newcommand{\edeff}{\end{definition}}
\newcommand{\brem}{\begin{remark}}
\newcommand{\erem}{\end{remark}}
\newcommand{\bex}{\begin{example}}
\newcommand{\eex}{\end{example}}

\newcommand{\bR}{\mathbb{R}}
\newcommand{\be}{\begin{equation}}
\newcommand{\ee}{\end{equation}}
\newcommand{\mS}{\mathbb{S}}

\newcommand{\lp}{\left(}
\newcommand{\rp}{\right)}
\newcommand{\lb}{\left[}
\newcommand{\rb}{\right]}
\newcommand{\lc}{\left\{}
\newcommand{\rc}{\right\}}

\newcommand{\Hess}{\text{Hess}}

\DeclareMathOperator{\Cut}{Cut}
\DeclareMathOperator{\Con}{Con}

\renewcommand{\phi}{\varphi}

\newcommand{\blue}[1]{\textcolor{blue}{#1}}

\newcommand{\bi}{\begin{itemize}}
\newcommand{\iii}{\item}
\newcommand{\ei}{\end{itemize}}
\newcommand{\bd}{\begin{description}}
\newcommand{\ed}{\end{description}}

\newcommand{\bqn}{\begin{eqnarray}}
\newcommand{\eqn}{\end{eqnarray}}
\newcommand{\eqnn}{\nonumber\end{eqnarray}}

\newcommand{\nn}{\nonumber}
\newcommand{\ba}[1]{\begin{array}{#1}}
\newcommand{\ea}{\end{array}}

\newcommand{\lam}{\lambda}

\newcommand{\g}{\gamma}

\newcommand{\eps}{\varepsilon}

\newcommand{\R}{\mathbb{R}}
\newcommand{\N}{\mathbb{N}}

\newcommand{\mc}[1]{\mathcal{ #1 }}

\newcommand{\all}{\forall\,}

\newcommand{\la}{\langle}
\newcommand{\ra}{\rangle}

\newcommand{\virg}[1]{``#1''}
\newcommand{\tx}[1]{\mathrm{#1}}
\newcommand{\til}[1]{\widetilde{#1}}

\newcommand{\distr}{\mc{D}}

\newcommand{\metr}{\textsl{g}}

\newcommand{\Pg}[1]{\left\{ #1 \right\}}

\newcommand{\hp}{hypothesis}
\newcommand{\EXP}{\mc{E}}

\newcommand{\Exp}{\mc{E}}
\newcommand{\lapl}{\Delta}
\newcommand{\dive}{\text{div}}

\newcommand{\HH}{\mc{H}}
\newcommand{\tcon}{t_{con}}

\input xy
 \xyoption{all}

\newcommand{\tcut}{t_{cut}}

\begin{document}
\begin{center} \noindent
{\LARGE{\sl{\bf On the heat diffusion for generic Riemannian and sub-Riemannian structures}}}

\vskip 0.6 cm
Davide Barilari, Ugo Boscain, Gr\'egoire Charlot, Robert W. Neel

\vskip 0.5cm
\today
\end{center}

\vskip 0.3 cm
\begin{abstract} 
In this paper we provide the small-time heat kernel asymptotics at the cut locus in three relevant cases:  generic low-dimensional Riemannian manifolds, generic 3D contact sub-Riemannian manifolds (close to the starting point) and generic 4D quasi-contact sub-Riemannian manifolds (close to a generic starting point). As a byproduct, we show that, for generic low-dimensional Riemannian manifolds, the only singularities of the exponential map, as a Lagragian map, that can arise along a minimizing geodesic are $A_3$ and $A_5$ (in the classification of Arnol'd's school).  We show that in the non-generic case, a cornucopia of asymptotics can occur, even for Riemannian surfaces.
\end{abstract}

{\small {\bf MSC classes}:	53C17 $\cdot$	57R45 $\cdot$ 58J35 }\\[-0.3cm]

{\small {\bf Keywords}: generic sub-Riemannian geometry, heat kernel asymptotics, singularity theory.}

\section{Introduction}

Let $M$ be a complete Riemannian or sub-Riemannian manifold.\footnote{In this paper, by sub-Riemannian manifold, we mean a constant rank sub-Riemannian manifold which is not Riemannian. By a (sub)-Riemannian manifold, we mean 
 a constant-rank sub-Riemannian manifold, which is possibly Riemannian. However, several results of the paper hold for more general rank-varying structures in the sense of \cite[Appendix A]{srneel}, for instance, Grushin-like structures. This will be specified in the paper.}
Endow this structure with a smooth volume $\mu$ and consider the associated Laplace operator $\lapl$ defined as the divergence of 
the (horizontal) gradient. Under the assumption that the H\"ormander condition is satisfied (which is obvious in the Riemannian case), 
and thanks to the completeness of $M$, the operator $\lapl$ is hypoelliptic and admits a smooth symmetric heat kernel $p_{t}(q_1,q_2)$ (see \cite{hormander,strichartz}).

The problem of relating the small-time asymptotics of $p_{t}(q_1,q_2)$ with the properties of the (sub)-Riemannian distance $d$ attracted 
attention starting from the 40s in the Riemannian case \cite{mina53,mina49}, and then starting from the 80s for the sub-Riemannian 
one \cite{benarous,leandremaj,leandremin}.
See \cite{berger,rosenberg} for a historical viewpoint on the Riemannian heat equation, and \cite{grigoryan} for a discussion about stochastic completeness of Riemannian manifolds.  

In this paper we are interested in this problem ``off-diagonal,'' namely when $q_1\neq q_2$. 
Let $\Cut(q_1)$ be the set of points where geodesics starting from $q_1$ lose their global optimality.
Notice that, when there are no abnormal minimizers, $\Cut(q_1)\cup\{q_1\}$ (without $\{q_1\}$ in the Riemannian case) corresponds to the set of points where the function $d^{2}(q_1,\cdot)$ is not smooth. 
In 1988  Ben-Arous \cite{benarous} proved that for $q_2\notin\Cut(q_1)\cup\{q_1\}$ and when there are no abnormal minimizers, one has the small-time heat kernel expansion (here $n$ is the dimension of the manifold)
\begin{equation}\label{eq:stimaCut}
p_t(q_1,q_2) = \frac{C+O(t)}{t^{n/2}}  e^{-d^{2}(q_1,q_2)/4t},
\end{equation}
for some $C>0$ depending on $q_1$ and $q_2$.
This extended known results  in the Riemannian context (see  Molchanov \cite{molcanov}, Azencott \cite{azencott}, Berger \cite{berger} and reference therein). 

The behavior of $p_{t}(q_1,q_2)$ in the case $q_2\in \Cut(q_1)$ was  a long-standing open problem. In Riemannian geometry it was investigated  
in \cite{molcanov}, with further development in \cite{neel,neelstroock,srneel}. 

In the sub-Riemannian case, few properties of the cut locus were known before the mid-90s. Motivated by the need to better understand the 
structure of the cut locus, arising from the work of Ben Arous and Leandre, a precise description of the local structure of the cut locus 
for  generic 3D contact structures and 4D generic quasi-contact structures was given by the control theory community, see \cite{agrexp,gauthiercontact,grisha-zaka,QUASI-CONTACT}.

However the connection between the structure of the singularity and the heat kernel asymptotics remained unclear. Recently substantial 
progress in this direction has been made in \cite{srneel}, by adapting and further developing the techniques of Molchanov \cite{molcanov}. 
In particular, in \cite{srneel} it was proved that the asymptotic is different from that of Equation \eqref{eq:stimaCut} if and only if $q_{2}$ is \emph{cut-conjugate} to $q_{1}$ (i.e. $q_2\in \Cut(q_{1})$ and $q_{2}$ is conjugate 
to $q_1$ along at least one optimal geodesic).\footnote{We recall that, in absence of abnormals, a geodesic cannot be conjugate before it reaches the cut locus.} In this case we have the bounds, for small $t$:
\begin{equation}\label{eq:stimaconj}
\frac{C}{t^{n/2+1/4}}  e^{-d^{2}(q_1,q_2)/4t} \leq p_t(q_1,q_2)
\leq\frac{C'}{t^{n-(1/2)}}  e^{-d^{2}(q_1,q_2)/4t},
\end{equation}
for some $C,C'>0$ depending on $q_1$ and $q_2$.
More precise estimates can be obtained if one has a precise description of ``how much'' $q_2$ is conjugate to $q_1$ (cf. Theorem \ref{t:zele}).

The main purpose of this paper is to continue this analysis by giving the precise leading term for three cases of interests

\bi
\iii[(i)] generic Riemannian manifolds of dimension less than or equal to 5 (starting from a fixed point).
\iii[(ii)] generic 3D contact sub-Riemannian manifolds (for $q_2$ close enough to $q_1$)
\iii[(iii)] generic 4D quasi-contact sub-Riemannian manifolds (for $q_2$ close enough to $q_1$, with generic $q_1$)
\ei
In the generic Riemannian case we stop at dimension 5 because up to dimension 5 the singularities of the exponential map are classified and they are structurally stable, see Section \ref{s-riemannian}. As a byproduct of our analysis we get that the list of possible singularities arising generically on minimizing geodesics in the 
Riemannian case is rather restricted, see  Theorem\ \ref{t:tgv0} below. Moreover we get a refinement of the main result \cite[Theorem 27]{srneel}. See Theorems \ref{t-ogrande} and \ref{t:zele}.

The results given for cases (ii) and (iii) answer the question mentioned above about the heat kernel asymptotics at the 
sub-Riemannian cut locus which originally motivated many efforts by the control theory community in sub-Riemannian geometry.

In a complementary direction, we show that a variety of aymptotics can occur even in a two-dimensional Riemannian manifold if one looks to non-generic cases.

For simplicity of exposition, to state our main results, we separate the Riemannian and sub-Riemannian cases.

\subsection{The Riemannian case}\label{s-riemannian}

In this paper we always assume that the Riemannian metrics are complete. In addition, the following definition specifies what we mean by a generic Riemannian metric. 
\bdeff
\label{d-generic}
Let $M$ be a smooth manifold and  $\mc{G}$ be the set of all complete Riemannian metrics on $M$ endowed with the $C^{\infty}$ Whitney topology.
When we say that \emph{for a generic Riemannian metric on a manifold $M$ the property} (P) \emph{is satisfied} we mean that the property (P) is satisfied on an open and dense subset of the set $\mc{G}$.
\edeff
In the following,  $\Exp_q$ denotes the exponential map based at the point $q\in M$, thought as a map that, with an initial direction $v$ (i.e., $v\in T_qM$, $|v|=1$) and a time $t$, associates the end point $\Exp_q(t v)$ of the arc-length parameterised geodesic defined on $[0,t]$, starting at $q$ with direction $v$.

Conjugate points along geodesics correspond to singularities of the exponential map, which is a Lagrangian\ffoot{\blue{COMMENT: do we define what is a lagrangian map? Davide: YES}} map. 
Thanks to the work of Arnol'd and his group (see \cite{Arnold} and references therein), we know the list of generic Lagrangian singularities
up to dimension 5, which coincides with the list of stable Lagrangian singularities up to dimension 5.

\bt[Arnol'd's school: normal form of generic singularities of Lagrangian maps
] \label{t:list}
Let $f$ be a generic Lagrangian map from $\R^n$ to $\R^n$ which is singular\ffoot{\blue{COMMENT:do we define what is singular??? Davide: NO}} at a point $q$. Then, there exist  changes of coordinates around $q$ and $f(q)$ such that in the new coordinates $q=0$ and:
\bi 
\item if $n=1$ then $f$ is the map\\  
$x\mapsto x^2$
(in this case the singularity is called $A_2$);
\item if $n=2$ then $f$ is the map\\
$(x,y)\mapsto(x^3+xy,y)$ ($A_3$)\\ or  a suspension\ffoot{\blue{COMMENT:do we have to define what is a suspension. Davide: YES}} of the previous one;
\item if $n=3$ then $f$ is the map\\
$(x,y,z)\mapsto(x^4+x^{2}y+xz,y,z)$ ($A_4$),\\
or $(x,y,z)\mapsto(x^2+y^2+xz,xy,z)$ ($D_4^+$),\\
or $(x,y,z)\mapsto(x^2-y^2+xz,xy,z)$ ($D_4^-$),\\
or  a suspension of the previous ones;
\item if $n=4$ then $f$ is the map\\
$(x,y,z,t)\mapsto(x^5+x^{3}y+x^{2}z+xt,y,z,t)$ ($A_5$),\\
or $(x,y,z,t)\mapsto(x^3+y^2+x^2z+xt,xy,z,t)$ ($D_5^+$),\\
or $(x,y,z,t)\mapsto(-x^3+y^2+x^2z+xt,xy,z,t)$ ($D_5^-$),\\
or  a suspension of the previous ones;
\item if $n=5$ then $f$ is the map\\
$(x,y,z,t,u)\mapsto(x^6+x^{4}y+x^{3}z+x^{2}t+xu,y,z,t,u)$ ($A_6$),\\
or $(x,y,z,t,u)\mapsto(x^4+y^2+x^3z+x^{2}t+xu,xy,z,t,u)$ ($D_6^+$),\\
or $(x,y,z,t,u)\mapsto(-x^4+y^2+x^3z+x^{2}t+xu,xy,z,t,u)$ ($D_6^-$),\\
or $(x,y,z,t,u)\mapsto(x^2+xyz+ty+ux,y^3+x^2z+tx,z,t,u)$ ($E_6^+$),\\
or $(x,y,z,t,u)\mapsto(x^2+xyz+ty+ux,-y^3+x^2z+tx,z,t,u)$ ($E_6^-$),\\
or  a suspension of the previous ones.
\ei
\et
Here by a suspension of a map $f:\R^{n}\to \R^{n}$ we mean a map of the form $\til{f}:\R^{n+n'}\to \R^{n+n'}$ defined by $\til{f}(x,y)=(f(x),y)$ where $x\in \R^{n}$ and $y\in \R^{n'}$.

It has  been long believed that singularities of generic Riemannian exponential maps are generic Lagrangian singularities.
 This was stated by Weinstein \cite{weinstein} and confirmed in Wall \cite{ctcwall}. A complete proof appears in \cite{polish} (see also \cite{klok}).


More precisely we have the following result.
\bt
\label{t:weinstein}
Let $M$ be a smooth manifold with $\dim M\leq 5$, and fix $q\in M$.
For generic Riemannian metrics on  $M$, the singularities of the exponential map $\Exp_{q}$ are those listed in Theorem \ref{t:list}.
\et

A preliminary question to the study of the heat kernel asymptotics at generic singularities is to understand 
which singularities of the Arnol'd list can be obtained as ``optimal'' ones. In other words, we consider singularities of the 
exponential map that can arise along minimizing geodesics. 

\bdeff Let $M$ be a Riemannian manifold, $q_1,q_2\in M$ such that $\gamma(t)=\Exp_{q_1}(tv)$ for $0\leq t\leq 1$ gives a minimizing conjugate geodesic 
from $q_1$ to $q_2$. Then we say that $\gamma$ is \emph{$A_2$-conjugate}  if at $v$, $\Exp_{q_1}$ has a normal form given by $A_2$. 
We define $A_3$-conjugacy, etc.\ in a similar way.
\edeff
In Section \ref{s-teorema-minimalita} we prove the following result.
\bt\label{t:tgv0}
Let $M$ be a smooth manifold, $\dim M\leq 5$, and $q_1\in M$. For a generic Riemannian metric on $M$  
and any minimizing geodesic $\gamma$ from $q_1$ to some $q_2$ we have that $\gamma$ is either non-conjugate, $A_{3}$-conjugate, 
or $A_{5}$-conjugate (at $q_2$).
\et
Notice that $A_{3}$ appears only for $n\geq 2$ and $A_{5}$ for $n\geq 4$.
As a consequence of Theorem \ref{t:tgv0} and of Theorem \ref{t-ogrande} in Section \ref{s-techniques} (which gives a finer analysis than that presented in \cite{srneel} of the relation between the degree of conjugacy and the small-time heat kernel asymptotics), we get

\bc\label{c:tgv1} Let $M$ be a smooth manifold, $ \dim M=n\leq 5$, and $q_1\in M$. For a generic Riemannian metric on $M$ the only possible heat kernel asymptotics are 
(here $C>0$ is some  constant which can differ from line to line):
\bi
\iii[(i)] If no minimizing geodesic from $q_1$ to $q_2$ is conjugate then
$$ p_{t}(q_1,q_2)=\frac{C+O(t)}{t^{\frac{n}{2}}} e^{-d^2(q_1,q_2)/4t},$$
\iii[(ii)] If at least one minimizing geodesic from $q_1$ to $q_2$ is $A_{3}$-conjugate but none is $A_{5}$-conjugate
$$p_{t}(q_1,q_2)=\frac{C+O(t^{1/2})}{t^{\frac{n}{2}+\frac{1}{4}}} e^{-d^2(q_1,q_2)/4t},$$
\iii[(iii)] If at least one minimizing geodesic from $q_1$ to $q_2$  is $A_{5}$-conjugate
$$p_{t}(q_1,q_2)=\frac{C+O(t^{1/3})}{t^{\frac{n}{2}+\frac{1}{6}}} e^{-d^2(q_1,q_2)/4t}.$$
\ei
\ec
Notice that this is consistent with the results obtained in \cite{jacek} on surfaces of revolution.

\subsection{The sub-Riemannian case} \label{s:srcase}
Some of the techniques developed to prove Theorems \ref{t:tgv0} and Corollary \ref{c:tgv1} can be adapted to the sub-Riemannian context. 

As in the Riemannian case, we always assume that the sub-Riemannian metrics are complete as metric spaces.\ffoot{\blue{COMMENT:If there are no abnormal minimisers, then completeness and geodesic completeness are equivalent, see \cite{nostrolibro}, VERIFY THE PRECISE STATEMENT.}} Moreover, when we say generic, we always refer to open and dense subsets with respect to the $C^{\infty}$  Whitney topology (see Definition \ref{d-generic}).\ffoot{\blue{COMMENT:do we have to specify that the distribution is fixed?}}

\subsubsection{3D contact case}
In contact geometry, there are no nontrivial abnormal extremals (for the definition of abnormals see Remark \ref{r:strict}). For the generic 3D contact case,  in \cite{agrexp,gauthiercontact} it is shown that close to the diagonal only singularities 
of type $A_{3}$ appear, accumulating to the initial point. The local structure of the first conjugate locus (see Fig. \ref{f-rd-4cusps-6cusps})\ffoot{\blue{COMMENT:MAYBE PUT A FIGURE OF THE SECTION}} is either a suspension 
of a four-cusp astroid (at generic points) or a suspension of a ``six-cusp astroid'' (along some special curves). For the four-cusp case, 
two of the cusps are reached by cut-conjugate geodesics, while in the six-cusp case this happens for three of them. 
Notice that the first conjugate locus at a generic point looks like a suspension of the first conjugate locus that one gets on a Riemannian ellipsoid \cite{jacek}.\ffoot{\blue{COMMENT:CITARE ANCHE ALTRE COSE SULL"ELLISSE}} 
The precise statement of these facts can be found  in \cite{agrexp,gauthiercontact} (see also \cite{nostrolibro}).

\begin{figure}
\begin{center}
\scalebox{0.9}{
\input{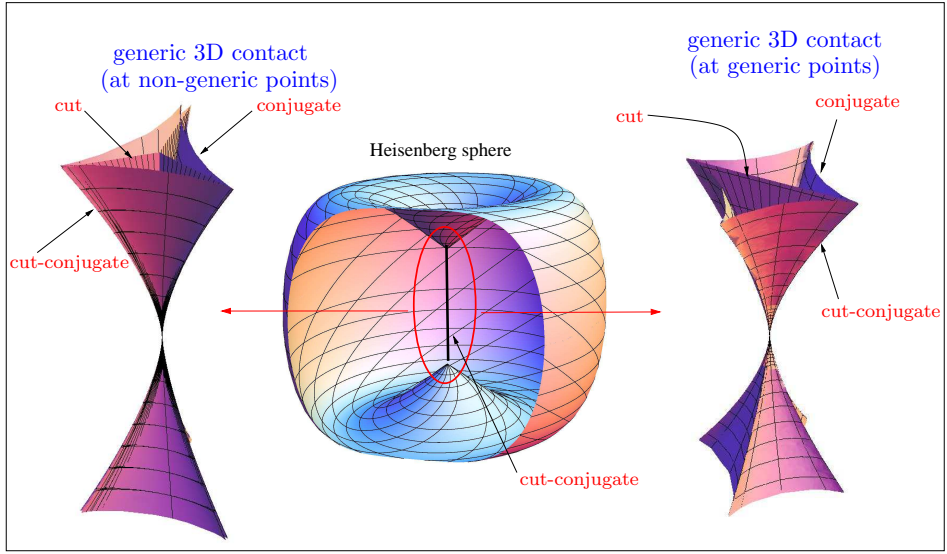}}
\label{f-rd-4cusps-6cusps}
\caption{Cut locus and first conjugate locus in a generic 3D contact  sub-Riemannian structure close to the starting point. 
The picture in the center shows the endpoints of optimal geodesics at time 1 for the Heisenberg group, which is the simplest 3D contact sub-Riemannian structure. The vertical axis is the cut locus, which coincides with the first conjugate locus. 
The picture on the right shows the cut locus and the first conjugate locus after a generic perturbation at a generic point. The first conjugate locus is a suspension of a four-cusp astroid. The picture on the left shows the cut locus and the first conjugate locus after a generic perturbation at a non-generic point. The first conjugate locus is a suspension of a ``six-cusp astroid.''}
\end{center}
\end{figure}

As a consequence of this analysis and of Theorem \ref{t-ogrande} in Section \ref{s-techniques} we get the following characterization of the heat kernel asymptotics for the 3D contact case.

\bc 
\label{c:sarrusofono1}
Let $M$ be a smooth manifold of dimension $3$.  Then for a generic 3D contact sub-Riemannian metric on $M$, every $q_1$, and every $q_2$ 
(close enough to $q_1$) we have that the only possible heat kernel asymptotics are (here $C>0$ is some  constant which can differ from line to line):
\bi
\iii[(i)] If no minimizing geodesic from $q_1$ to $q_2$ is conjugate then
$$ p_{t}(q_1,q_2)=\frac{C+O(t)}{t^{3/2}} e^{-d^2(q_1,q_2)/4t},$$
\iii[(ii)] If at least one minimizing geodesic from $q_1$ to $q_2$ is conjugate then
$$p_{t}(q_1,q_2)=\frac{C+O(t^{1/2})}{t^{7/4}} e^{-d^2(q_1,q_2)/4t}.$$
\ei
Moreover, there are points $q_2$ arbitrarily close to $q_1$ such that case (ii) occurs. 
\ec

Notice that exponents of the form $N/4$, for integer $N$, were unexpected in the 90s literature, see \cite{leandre}.

\subsubsection{4D quasi-contact case}

For generic 4D quasi-contact metrics there are no non-trivial abnormal minimizers \cite{agrachevjp} (for non-generic structures, 
nontrivial abnormal minimizers can exist but they are not strict).  The structure of the local first conjugate locus is known for generic points, see \cite{QUASI-CONTACT}. 
More precisely, outside of a stratified set of codimension 1, the singularities corresponding 
to the first conjugate locus are of type $A_2$, $A_3$, and $D_4^+$. Moreover, among the points of the first conjugate locus, 
those that are cut-conjugate correspond to $A_3$ singularities, and they accumulate to the initial point.

Thanks to this analysis and to Theorem \ref{t-ogrande} (in Section \ref{s-techniques}) we get
\bc 
\label{c:sarrusofono2}
Let $M$ be a smooth manifold of dimension $4$.  Then, for a generic quasi-contact sub-Riemannian metric on $M$,
for $q_1$ outside a stratified set of codimension 1 and $q_2$ close enough to $q_1$, the only possible heat kernel asymptotics are 
(here $C>0$ is some constant which can differ from line to line):
\bi
\iii[(i)] If no minimizing geodesic from $q_1$ to $q_2$ is conjugate then
$$ p_{t}(q_1,q_2)=\frac{C+O(t)}{t^2} e^{-d^2(q_1,q_2)/4t},$$
\iii[(ii)] If at least one minimizing geodesic from $q_1$ to $q_2$ is conjugate then 
$$p_{t}(q_1,q_2)=\frac{C+O(t^{1/2})}{t^{9/4}}e^{-d^2(q_1,q_2)/4t}.$$
\ei
Moreover, there are points $q_2$ arbitrarily close to $q_1$ such that case (ii) occurs. \ffoot{work on the quasi contact nilpotent case for the abnormal}


\ec

\subsection{Techniques}
\label{s-techniques}
The analysis of the heat kernel asymptotics in the three cases of
interest just discussed consists of a few steps. First, the structure
of the minimizing geodesics must be understood. The possible generic
singularities of the exponential map for the 3D contact and 4D
quasi-contact cases were essentially already known, and we have
already summarized those results. In the Riemannian case, this means
that we wish to prove Theorem \ref{t:tgv0}, which is of independent
interest (in particular, compare the list of possible singularities
occurring for \emph{minimizing} geodesics with the much longer list of
singularities that can occur for geodesics past the point of being
minimizing, as given in Theorems \ref{t:list} and \ref{t:weinstein}).
To do this, we consider the hinged energy function (see \cite{molcanov,neel,srneel})
\[
h_{q_1,q_2}(q) = \frac{1}{2}\lp d^2\lp q_1,q \rp + d^2\lp q,q_2 \rp\rp.
\]
This function achieves its minimum of $d^2(q_1,q_2)/4$ precisely at
the midpoints of minimizing geodesics from $q_1$ to $q_2$ \cite[Lemma 21]{srneel}. Further, the
local structure of $h=h_{q_1,q_2}$ near these midpoints is closely
related to the structure of the exponential map along the
corresponding geodesic. In particular, because $h$ has a minimum at
such a midpoint, its ``normal form'' is more restricted than at an
arbitrary critical point (for example, in dimension two, $h$ cannot have a normal form of the type $x^2+y^3$).
 Transferring these restrictions from $h$ to the
exponential map allows us to show that most of the singularities in
Theorem \ref{t:list} cannot occur at minimizing geodesics. The details
are given in Section \ref{s-technical}.

Once we have a description of the possible singularities that can
appear along a minimizing geodesic for a given generic geometry, we then
determine the corresponding normal form for the singularity of $h$ at
the midpoint of such a geodesic. 
Having determined the possible
structures of $h$, 
we determine the corresponding heat-kernel asymptotics in the spirit of \cite{srneel}. 
However we actually need a finer result than was given there.

This is of some independent interest, so we elaborate on it
here. We will consider minimizing geodesics such that the exponential map
has a singular point of type $(1,m)$ (see Definition \ref{d:sing}); such
geodesics include those which are $A_m$-conjugate. Here $h$ has the
simple form $d^2(q_1,q_2)/4+x_1^2+\cdots+x_{n-1}^2+x_n^{m+1}$ with $m$ odd, in some
coordinates around the midpoint of the geodesic (see Lemma
\ref{l:hh}). The corresponding heat kernel asymptotics are described
in the following theorem, which gives a refinement of (part of)
Theorem 27 from \cite{srneel}. This also improves upon the pioneering work of Molchanov, see \cite[Subcase 3.2.a]{molcanov}.

\bt
\label{t-ogrande}
Let $M$ be an $n$-dimensional, complete (sub)-Riemannian manifold, and let $q_1$ and $q_2$ be distinct points such that
all minimizing geodesics from $q_1$ to $q_2$ are strongly normal. Suppose
that, for some $\ell\in \{3,5,7,\ldots\}$ every minimizing geodesic from
$q_1$ to $q_2$ corresponds to a singular point of type $(1,m)$ with
$3\leq m\leq \ell$, and that there is at least one minimizing geodesics
where the singular point is of type $(1,\ell)$. Then there exists
$C>0$ such that
\[
p_t(q_1,q_2) = \frac{C+O\lp
t^{\frac{2}{\ell+1}}\rp}{t^{\frac{n+1}{2}-\frac{1}{\ell+1}}}e^{-d^2(q_1,q_2)/4t}.
\]
\et

\brem
Theorem \ref{t-ogrande} holds for (sub)-Riemannian structures as defined in Section \ref{s:srg}, but also for more general rank-varying sub-Riemannian structures in the sense of \cite[Appendix A]{srneel}.
\erem

\subsection{Complementary results and non-generic cases}

As a further refinement of Theorem 27 from \cite{srneel}, we also
offer the following theorem (which, although not used elsewhere in the
present paper, helps to shed additional light on the relationship
between the structure of the minimizing geodesics and the small-time heat
kernel asymptotics). 

Recall that in the (sub)-Riemannian context, arclength geodesics starting from a point $q$ are parametrized by initial covectors $\lambda\in T^\ast_qM\cap\{H=1/2\}$  where $H$ is the (sub)-Riemannian Hamiltonian $H(q,\lambda)=\sum_1^k \la\lambda,X_i(q)   \ra^2$.  (Here $X_1,\ldots,X_k$ is an orthonormal frame for the structure.) Hence the exponential map $\Exp_q$ is a map that, with  $\lambda\in T^\ast_qM\cap\{H=1/2\}$ and a time $t$, associates   the end point $\Exp_q(t \lambda)$ of the arclength parameterised geodesic defined on $[0,t]$, starting at $q$, and corresponding to $\lambda$. Recall that $\Exp_q$ depends only on the product $t\lambda$ and not on $t$ and $\lambda$ separately. Hence $\Exp_q$ can be thought as defined on the whole of $T^\ast_qM$.

First, suppose that $q_1$ and $q_2$ are distinct
points on $M$, and let $\lambda\in T^\ast_{q_1}M$ be such that
$\Exp_{q_1}(2t\lambda)$
for $0\leq t\leq 1$ gives a minimizing geodesic from $q_1$ to $q_2$. Then
recall from Theorem 24 of \cite{srneel} that $d\Exp_{q_1}$ at $2\lambda$ has
rank $n-r$ if and only if the Hessian of $h_{q_1,q_2}$ at $\Exp_{q_1}(\lambda)$ has rank $n-r$. The idea is
that this rank partially describes how conjugate this minimizing geodesic
is by indicating the number of independent perturbations with respect
to which it is conjugate.

\bt\label{t:zele}
Let $M$ be an $n$-dimensional complete 
(sub)-Riemannian  manifold, and let $q_1$ and $q_2$ be distinct points such that
there is a unique minimizing strongly normal geodesic  from $q_1$ to $q_2$ (which we
denote $\Exp_{q_1}(2t\lambda)$ for $0\leq t\leq 1$, as above). Then if $D_{2\lam}\Exp_{q_1}$ has rank
$n-r$ for some $r\in\{0,1,2,\ldots,n-1\}$, there exist $C_1, C_2, t_0
>0$ such that for $0<t<t_0$
\[
\frac{C_1}{t^{\frac n2+\frac{r}4}}e^{-d^2(q_1,q_2)/4t} \leq p_t(q_1,q_2) \leq
\frac{C_2}{t^{\frac n2+\frac{r}2}}e^{-d^2(q_1,q_2)/4t}.
\]
\et
\brem
As for Theorem \ref{t-ogrande}, Theorem \ref{t:zele} holds for (sub)-Riemannian structures as defined in Section \ref{s:srg}, but also for more general rank-varying sub-Riemannian structures in the sense of \cite[Appendix A]{srneel}.
\erem

The previous two theorems (namely Theorem \ref{t-ogrande} and Theorem
\ref{t:zele}) present an especially appealing picture when
considered together. Simply knowing the rank of the exponential map,
even when there is only a single minimizing geodesic, does not completely
determine the asymptotic behavior. Instead, the precise behavior in
the degenerate directions matters, as Theorem \ref{t-ogrande} indicates.
Indeed, both $A_3$ and $A_5$-conjugacy fall into the $r=1$ case of
Theorem \ref{t:zele}. Further, one could consider the
contribution to the small-time asymptotics of minimizing geodesics
exhibiting other singularities from the Arnol'd classification (Theorem
\ref{t-ogrande} dealing with the case when the only singularities are
from the $A_m$-family), and
determine which ones dominate which other ones. However, the situation
becomes more complicated as more types of singularities are allowed.
Indeed, the virtue of the low-dimensional cases considered here is
that the list of possible singularities which can occur, at least
generically, is sufficiently simple so as to allow a complete analysis.

On the other hand, if one is willing to go beyond the generic case, a
wide variety of expansions are possible, even for Riemannian metrics
in low dimensions. To illustrate this, we give the following theorem,
highlighting the contrast between the generic and non-generic
situations (compare this with Corollary \ref{c:tgv1}).

\bt\label{t:examples}
For any integer $\eta\geq 3$, any positive real $\alpha$, and any real
$\beta$, there exists a smooth metric on the sphere $\mS^2$ and (distinct) points
$q_1$ and $q_2$ such that the heat kernel has the small-time asymptotic
expansion
\[
p_t(q_1,q_2) = e^{-d^{2}(q_1,q_2)/4t}\frac{1}{t^{(3\eta-1)/2\eta}} \lc \alpha +
t^{1/\eta}\beta+o\lp t^{1/\eta}\rp\rc .
\]
\et

Note that both the leading term and the next term in this
expansion are precisely what one gets by applying Theorem
\ref{t-ogrande} to a surface and letting $\ell=2\eta-1$. Thus, the
existence of such expansions is not surprising. Nonetheless, the point
of Theorem \ref{t:examples} is to show that such singularities do in
fact occur (in the proof of Theorem \ref{t:examples}, we will see that
the metric we construct does yield a singularity of type $(1,m)$ for
appropriate $m$), and further, that the ``big-O'' term we see for the
next term in the expansion cannot in general be improved upon. That
is, we do see expansions in fractional powers of $t$, as Theorem
\ref{t-ogrande} seems to suggest.
\section{(Sub)-Riemannian geometry}\label{s:srg}
We start by recalling the definition of a (sub)-Riemannian manifold (including also Riemannian manifolds) in the case  of a distribution of constant rank. 
\bdeff
A \emph{(sub)-Riemannian manifold} is a triple $(M,\distr,\metr)$,
where
\bi
\iii[$(i)$] $M$ is a connected, orientable, smooth manifold of dimension
$n$.
\iii[$(ii)$] $\distr$ is a smooth vector distribution of constant rank $k\leq n$
satisfying the \emph{H\"ormander condition}, that is,  a smooth map that
associates with $q\in M$  a $k$-dimensional subspace $\distr_{q}$ of
$T_qM$ such that
\bqn \label{Hor}
\qquad \text{span}\{[X_1,[\ldots[X_{j-1},X_j]]]_{q}~|~X_i\in\overline{\distr},\,
j\in \N\}=T_qM, \ \all q\in M,
\eqn
where $\overline{\distr}$ denotes the set of \emph{horizontal smooth
vector fields} on $M$, i.e.\
$$\overline{\distr}=\Pg{X\in\mathrm{Vec}(M)\ |\ X(q)\in\distr_{q}~\
\forall~q\in M}.$$
\iii[$(iii)$] $\metr_q$ is a Riemannian metric on $\distr_{q}$ which is smooth
as function of $q$. We denote  the norm of a vector $v\in \distr_{q}$
by 
$|v|_{\metr}=\sqrt{\metr_{q}(v,v)}.$
\ei
\edeff

\brem
Notice that, when $k=n$, we are in the Riemannian context and the H\"ormander condition is automatically satisfied.
As a consequence of the definition, (sub)-Riemannian structures that are not Riemannian exist only in dimension $n\geq3$.
\erem

A Lipschitz continuous curve $\g:[0,T]\to M$ is said to be
\emph{horizontal} (or \emph{admissible}) if
$$\dot\g(t)\in\distr_{\g(t)}\qquad \text{ for a.e.\ } t\in[0,T].$$

Given a horizontal curve $\g:[0,T]\to M$, the {\it length of $\g$} is
\bqn
\label{e-lunghezza}
\ell(\g)=\int_0^T |\dot{\g}(t)|_{\metr}~dt.
\eqn
Notice that $\ell(\g)$ is invariant under time reparametrization of
the curve $\g$.
The {\it distance} induced by the (sub)-Riemannian structure on $M$ is the
function
\bqn
\label{e-dipoi}
d(q_0,q_1)=\inf \{\ell(\g)\mid \g(0)=q_0,\g(T)=q_1, \g\ \mathrm{horizontal}\}.
\eqn
The \hp\ of connectedness of $M$ and the H\"ormander condition
guarantee the finiteness and the continuity of $d(\cdot,\cdot)$ with
respect to the topology of $M$ (Chow-Rashevsky theorem, see for
instance \cite{agrachevbook}). The function $d(\cdot,\cdot)$ is called
the \emph{Carnot-Caratheodory distance} and gives to $M$ the structure
of a metric space (see for instance \cite{nostrolibro}).

Locally, the pair $(\distr,\metr)$ can be given by assigning a set of
$k$ smooth vector fields that span $\distr$ and that are orthonormal
for $\metr$, i.e.\
\bqn
\label{trivializable}
\distr_{q}=\text{span}\{X_1(q),\dots,X_k(q)\}, \qquad
\metr_q(X_i(q),X_j(q))=\delta_{ij}.
\eqn
In this case, the set $\Pg{X_1,\ldots,X_k}$ is called a \emph{local
orthonormal frame} for the sub-Riemannian structure.

A sub-Riemannian manifold of odd dimension is said to be \emph{contact} if 
$\distr= \ker \omega$, where $\omega \in \Lambda^{1}M$ and $d\omega|_{\distr} $ is non degenerate. 
A sub-Riemannian manifold $M$ of even dimension is said to be \emph{quasi-contact} if 
$\distr= \ker \omega$, where $\omega \in \Lambda^{1}M$ and satisfies $\tx{dim }\ker d\omega|_{\distr}=1$.

The sub-Riemannian metric can also be expressed locally in ``control
form'' as follows. We consider the control system,
\bqn\label{eq:lnonce0}
\dot q=\sum_{i=1}^m u_i X_i(q)\,,~~~u_i\in\R\,,
\eqn
and the problem of finding the shortest curve that joins two fixed
points $q_0,~q_1\in M$ is naturally formulated as the optimal control
problem
\bqn \label{eq:lnonce}
~~~\int_0^T  \sqrt{
\sum_{i=1}^m u_i^2(t)}~dt\to\min, \ \ \qquad q(0)=q_0,~~~q(T)=q_1\neq q_{0}.
\eqn

\subsection{Minimizers and geodesics}
 In this section we briefly recall some facts about (sub)-Riemannian
geodesics. In particular, we define the
(sub)-Riemannian exponential map.

 \bdeff A \emph{geodesic} for a (sub)-Riemannian manifold
$(M,\distr,\metr)$ is an admissible curve $\g:[0,T]\to M$ such that
 $|\dot{\g}(t)|_{\metr}$  is constant and, for every sufficiently
small interval $[t_1,t_2]\subset [0,T]$, the restriction
$\g_{|_{[t_1,t_2]}}$ is a minimizer of $\ell(\cdot)$.
A geodesic for which $|\dot{\g}(t)|_{\metr}=1$ is said to be
parametrized by arclength.
A (sub)-Riemannian manifold is said to be \emph{complete}
if $(M,d)$ is complete as a metric space.
\edeff
Notice that, if the (sub)-Riemannian
metric is the restriction to $\distr$ of a complete Riemannian metric,
then it is complete.
Under the assumption that the manifold is complete, a
version of the Hopf-Rinow theorem (see \cite{strichartz} or \cite[Chapter 2]{burago})
implies that the manifold is geodesically complete (i.e. all geodesics are defined for every $t\geq0$) and that for every two points there exists a minimizing geodesic
connecting them.

Trajectories minimizing the distance between two points are solutions
of first-order necessary conditions for optimality, which in the case
of (sub)-Riemannian geometry are given by a weak version of the
Pontryagin Maximum Principle (\cite{pontrybook}).
\bt\label{t:pmpw}
Let $q(\cdot):t\in[0,T]\mapsto q(t)\in M$ be a solution of the
minimization problem \eqref{eq:lnonce0},\eqref{eq:lnonce} such that
$|\dot q(t)|_{\metr}$  is constant and $u(\cdot)$ be the
corresponding control.
Then there exists a Lipschitz map $p(\cdot): t\in [0,T] \mapsto
p(t)\in T^{*}_{q(t)}M\setminus\{0\}$  such that one and only one of
the following conditions holds:
\bi
\iii[(i)]
$
\dot{q}=\dfrac{\partial H}{\partial p}, \quad
\dot{p}=-\dfrac{\partial H}{\partial q}, \quad
u_{i}(t)=\la p(t), X_{i}(q(t))\ra,
\\$
where $H(q,p)=\frac{1}{2} \sum_{i=1}^{k} \la p,X_{i}(q)\ra^{2}$.
\vspace{0.2cm}
\iii[(ii)]
$
\dot{q}=\dfrac{\partial \HH}{\partial p}, \quad
\dot{p}=-\dfrac{\partial \HH}{\partial q}, \quad
0=\la p(t), X_{i}(q(t))\ra,
\\$
where $\HH(t,q,p)=\sum_{i=1}^{k}u_{i}(t) \la p,X_{i}(q)\ra$.
\ei
\et
\noindent
For an elementary proof of Theorem \ref{t:pmpw} see \cite{nostrolibro}.

\brem \label{r:strict} If $(q(\cdot),p(\cdot))$ is a solution of (i) (resp.\ (ii)) then
it is called a \emph{normal extremal} (resp.\ \emph{abnormal
extremal}). It is well known that if $(q(\cdot),p(\cdot))$ is a normal
extremal then $q(\cdot)$ is a geodesic (see
\cite{nostrolibro,agrachevbook}). This does not hold in general for
abnormal extremals. An admissible trajectory $q(\cdot)$ can be at the
same time normal and abnormal (corresponding to different covectors).
If an admissible trajectory $q(\cdot)$ is normal but not abnormal, we
say that it is \emph{strictly normal}. Abnormal extremals do not exist in the Riemannian case.
Abnormal extremals are very difficult to treat and many questions are
still open. For instance, it is not known if abnormal minimizers are
smooth (see \cite{montgomerybook}).
\erem

\bdeff
A minimizer $\gamma:[0,T]\to M$ is said to be \emph{strongly normal}
if for every $[t_1,t_2]\subset [0,T]$, $\gamma|_{[t_1,t_2]}$ is not an
abnormal.\ffoot{COMMENT: should we say that actually we need only half curve that is not abnormal?}\edeff

In the following, we denote by $(q(t),p(t))=e^{t\vec{H}}(q_{0},p_{0})$
the solution of (i) of Theorem \ref{t:pmpw} with initial condition
$(q(0),p(0))=(q_{0},p_{0})$. Moreover we denote by $\pi:T^{*}M\to M$
the canonical projection.

Normal extremals (starting from $q_{0}$) parametrized by arclength
correspond to initial covectors $p_{0}\in \Lambda_{q_{0}}:=\{p_{0}\in
T^{*}_{q_{0}}M | \, H(q_{0},p_{0})=1/2\}.$
The \emph{exponential map} starting from $q_{0}$
is defined as
\bqn \label{eq:expmap}
\EXP_{q_{0}}: \R^{+}  \times\Lambda_{q_{0}}\to M, \qquad
\EXP_{q_{0}}(t p_{0})= \pi(e^{t\vec{H}}(q_{0},p_{0})).
\eqn
Next, we recall the definition of cut and first conjugate times.

\bdeff \label{def:cut} Let $q_{0}\in M$ and
$\g(t)$ an arclength geodesic starting from $q_{0}$.
The \emph{cut time} for $\g$ is $\tcut(\g)=\sup\{t>0:\, \g|_{[0,t]}
\text{ is optimal}\}$. The \emph{cut locus} from $q_{0}$ is the set
$\Cut(q_{0})=\{\g(\tcut(\g)): \g$  is an arclength geodesic from
$q_{0}\}$.
\edeff

\bdeff \label{def:con} Let $q_{0}\in M$ and
$\g(t)$ a normal arclength geodesic starting from $q_{0}$ with initial
covector $p_{0}$. Assume that $\g$ is not abnormal.
The \emph{first conjugate time} of $\g$ is
$\tcon(\g)=\min\{t>0:\  (t,p_{0})$ is a critical point of
$\EXP_{q_{0}}\}$. The \emph{first conjugate locus} from $q_{0}$ is the
set $\Con(q_{0})=\{\g(\tcon(\g)): \g$  is an arclength geodesic from
$q_{0}\}$.

\edeff
It is well known that, for a geodesic $\g$ which is not abnormal, the
cut time $t_{*}=\tcut(\g)$ is either equal to the conjugate time or
there exists another geodesic $\til{\g}$ such that
$\g(t_{*})=\til{\g}(t_{*})$ (see for instance \cite{agrexp}).

\brem
For sub-Riemannian manifolds (that are not Riemannian), the exponential map starting from $q_{0}$ is never a local
diffeomorphism in a neighborhood of the point $q_{0}$ itself. As a consequence the
sub-Riemannian balls are never smooth and both the cut and the
conjugate loci from $q_{0}$ are adjacent to the point $q_{0}$ itself
(see \cite{agratorino}).
\erem

\subsection{The (sub)-Laplacian}\label{s:lapl}
In this section we define the (sub)-Riemannian Laplacian
 on a (sub)-Riemannian manifold $(M, \distr, \metr)$, provided with a
smooth volume $\mu$.
The (sub)-Laplacian is the natural generalization of the
Laplace-Bel\-tra\-mi operator on a Riemannian manifold, defined as
the divergence of the (horizontal) gradient (see for instance \cite{laplacian}).
In a local orthonormal frame $X_{1},\ldots,X_{k}$ it can be written as follows
\bqn \label{eq:lapldiv}
\lapl= \sum_{i=1}^{k}X_{i}^{2}+ (\dive\, X_{i}) X_{i}.
\eqn
Notice that $\lapl$ is always expressed as the sum of squares of the
elements of the orthonormal frame, plus a first-order term that belongs
to the distribution and depends on the choice of the volume $\mu$.

Thanks to the H\"ormander condition, the sub-Laplacian is hypoelliptic \cite{hormander}.  The existence of a smooth heat kernel for the operator $\lapl$, in the case of a  complete sub-Riemannian manifold, is stated in \cite{strichartz}.
 
 \brem
 We recall that in the equiregular case, it is possible to define in an intrinsic way a volume form on the (sub)-Riemannian manifold (called Popp's volume \cite{montgomerybook}, which coincides with the Riemannian volume in the Riemannian case). When  the (sub)-Laplacian is computed with respect to Popp's volume it is called the ``intrinsic-Laplacian''. Other intrinsic non-equivalent volumes can be defined in certain cases, see \cite{corank1}.
 \erem

\section{Technical lemmas}\label{s-technical}
We first recall two technical results that we need. 
The first one is the \virg{splitting lemma} for smooth functions (which can be found in \cite{gromollmeyer}) which allows us to split off non-degenerate directions and thus partially diagonalize $g$.

\bl[\cite{gromollmeyer}] \label{l:splitting}
Let $g:U\subset \R^{n}$ be a smooth 
 function on a neighborhood $U$ of the origin in $\R^n$.  Assume that $g(0)=dg(0)=0$ and that $0$ is an isolated local minimum of $g$ with $\dim \ker d^2g(0)=k$. 
 
 Then there exists a diffeomorphism $\psi$ from a neighborhood of the origin in $\R^{n}$ to a neighborhood of the origin in $\R^{n}$ and a smooth function $\phi:\R^{k}\to \R$ such that
$$g(\psi(u))=u_1^2+\ldots+u_{n-k}^{2}+\phi(u_{n-k+1},\ldots,u_{n}).$$
\el

The second result relates the structure of minimizing geodesics from $q_1$ to $q_2$ to the behaviour of $h_{q_1,q_2}$ near its minimum.
Suppose we have
distinct points $q_1$ and $q_2$ such that every minimizer from $q_1$ to $q_2$
is strongly normal.  Let $\Gamma$ be the set of midpoints
of minimizing geodesics from $q_1$ to $q_2$.

Consider any point $z_0\in \Gamma$, which is the midpoint of some minimizing 
geodesic $\gamma$ from $q_1$ to $q_2$.  Then there is a unique covector
$\lambda\in T^*_{q_1}M$ such that $\Exp_{q_1}(\lambda)={z_0}$,
$\Exp_{q_1}(2\lambda)={q_2}$ and that $\Exp_{q_1}(2 t\lambda)$ for
$t\in[0,1]$ parametrizes $\gamma$.   

Let $\lambda(s)$ be a smooth curve of covectors $\lambda:(-\eps, \eps)
\rightarrow T^{*}_{q_1} M$ (for some small $\eps>0$) such that
$\lambda(0)=\lambda$ and the derivative never vanishes.  Thus
$\lambda(s)$ is a one-parameter family of perturbations of $\lambda$
which realizes the first-order perturbation $\lambda^{\prime}(0) \in
T_{\lambda}\lp T^{*}_{q_1} M\rp$.    Also, we let $z(s)= \Exp_{q_1}
(\lambda(s))$, so that $z(0)=z_0$.  Because $\Exp_{q_1}$ is a
diffeomorphism from a neighborhood of $\lambda$ to a neighborhood of
$z_0$, we see that the derivative of $z(s)$ also never vanishes.  Thus
$z(s)$ is a curve which realizes the vector $z^{\prime}(0)\in
T_{z_0}M$.  Further, we've established an isomorphism of the vector
spaces $T_{\lambda}\lp T^{*}_{q_1}  M\rp$ and $T_{z_0}M$ by mapping
$\lambda^{\prime}(0)$ to $z^{\prime}(0)$, except that we've excluded
the origin by insisting that both vectors are non-zero.

We say that $\gamma$ is conjugate in the direction
$\lambda^{\prime}(0)$ (or with respect to the perturbation
$\lambda^{\prime}(0)$) if $\frac{d}{ds}\Exp_{q_1}(2\lambda(s))|_{s=0} =0$.
Note that this only depends on $\lambda^{\prime}(0)$.  We say that the
Hessian of $h_{{q_1},{q_2}}$ at $z_0$ is degenerate in the direction
$z^{\prime}(0)$ if $\frac{d^2}{ds^2}h_{{q_1},{q_2}}(z(s))|_{s=0} =0$.  This
last equality is equivalent to writing the Hessian of $h_{{q_1},{q_2}}$ as a
matrix in some smooth local coordinates, applying it as a quadratic
form to $z^{\prime}(0)$ expressed in these coordinates, and getting
zero.  This equivalence, as well as the fact that whether the result is
zero or not depends only on $z^{\prime}(0)$, follows from the fact
that $z_0$ is a critical point of $h_{{q_1},{q_2}}$.

The point of the  next theorem is that conjugacy in the direction
$\lambda^{\prime}(0)$ is equivalent to degeneracy in the direction
$z^{\prime}(0)$.

\bt[\cite{srneel}]\label{t:handconj}
Let $M$ be a  complete (sub)-Riemannian
manifold, and let $q_{1}$ and $q_{2}$ be distinct
points such that every minimizing geodesic from $q_1$ to $q_2$ is strongly
normal. Let $\gamma$ be a minimizing geodesic from $q_{1}$ to $q_{2}$, and define $\Gamma$, $z_0\in \Gamma$, $h_{{q_1},{q_2}}$ and the curves $\lam(s), z(s)$ as above.
Then  
\bi
\iii[(i)]  $\gamma$ is conjugate
if and only if the Hessian of
$h_{{q_1},{q_2}}$ at $z_0$ is degenerate. 
\iii[(ii)] In particular $\gamma$ is
conjugate in the direction $\lambda^{\prime}(0)$ if and only if the
Hessian of $h_{{q_1},{q_2}}$ at $z_0$ is degenerate in the corresponding
direction $z^{\prime}(0)$. 
\iii[(iii)] The dimension of the kernel of $D_{2\lam}\Exp_{q_1}$ is equal to the dimension of the kernel of the
Hessian of $h_{{q_1},{q_2}}$ at $z_0$. 
\ei
\et

We now need the following

\bdeff \label{d:sing} Let $f:M\to N$ be a smooth map between two $n$-dimensional manifolds and $q\in M$. We say that $q$ is a \emph{singular point of type $(1,m)$} for $f$ if 
\bi
\iii[(i)] $\tx{rank}(D_{q}f)=n-1$.
\iii[(ii)] in a system of coordinates we have
$$m=\max\{ k\in \N \,|\, f(\gamma(t)) = f(q)+t^{k}v+o(t^{k}), v\neq 0,  \gamma\in \mc{C}_{q}\}$$
where $\mc{C}_{q}$ is the set of smooth curves $\gamma$ such that $\gamma(0)=q$ and $\gamma'(0)\neq0$.
\ei
\edeff
Notice that in this definition one necessarily has $m\geq 2$. The next two lemmas are crucial in what follows.

\bl\label{l:hh}
Let $q_1$ and $q_2$ be conjugate along a minimizing geodesic $\gamma(t)=\EXP_{q_1}(2t\lambda)$ for $t\in[0,1]$, and let $z_{0}$ be the midpoint.  
Then $2\lam$ is a singular point of type $(1,m)$ for $\EXP_{q_1}$ if and only if there exists a system of coordinates $(x_1,\dots,x_n)$ 
around $z_{0}$ such that 
\bqn\label{eq:hhh}
h_{q_1,q_2}(x)= \frac{1}{4}d^2(q_1,q_2)+ x_{1}^{2}+\ldots+x_{n-1}^{2}+x_{n}^{m+1}.
\eqn
\el

\begin{proof} We prove only that if $2\lam$ is a singular point of type $(1,m)$ for $\EXP_{q_1}$ then there exists coordinates such that \eqref{eq:hhh} holds. The converse is similar.

Thanks to Theorem \ref{t:handconj}, we have that $\dim \ker D_{2\lam}\EXP_{{q_1}}=1$ 
if and only $\dim \ker \Hess_{z_{0}} h_{{q_1},{q_2}}=1$. Hence, by Lemma \ref{l:splitting}, 
there exists a coordinate system at $z_{0}=\EXP_{q_1}(\lambda)$ such that 
$h_{{{q_1},{q_2}}}(x_1,\dots,x_n)=h(z_0)+x_1^2+\dots+x_{n-1}^2+\varphi(x_n)$ with $\varphi$ smooth. Since $h(z_0) =d^2(q_1,q_2)/4$, we are left to check that we can choose the coordinate system in such a way that $\varphi(x_n)=x_{n}^{m+1}$.

Consider any smooth curve $s\mapsto\lambda(s)$ in $T_{q_1}^*M$ such that $\lambda(0)=\lambda$, and write
$x(s)=\EXP_{q_1}(\lambda(s))$ and ${y}(s)=\EXP_{q_1}(2\lambda(s))$. It is clear that $x(0)=z_0$ and ${y}(0)={q_2}$.
Using that $z_{0}$, the midpoint of the curve, does not belong to the cut locus of ${q_2}$, one can prove (see the proof of Theorem \ref{t:handconj} in \cite{srneel}) that the families $d_{x(s)}h_{{{q_1},{q_2}}}$ and $y(s)-q_2$ have the same order with respect to $s$ (in any coordinate system). 
In particular, if we choose the family $\lambda(s)$ corresponding to the curve $x(s)=(0,\dots,0,s)$ (this is possible since the exponential map is a diffeomorphism near these points), then
 $d_{(0,\dots,0,s)}h_{{{q_1},{q_2}}}=\varphi'(s)dx_n$. Now assume that $\varphi(x_{n})=c x_{n}^{k+1}+o(x_{n}^{k+1})$ with $c\neq 0$. Up to a
reparameterization of the $x_n$-axis, we can assume that $\varphi(x_{n})=x_{n}^{k+1}$. We want to show that $k=m$. Indeed it is easy to show that 
for every choice of a curve $\lambda(s)$ as before, $d_{x(s)}h_{{{q_1},{q_2}}}$ has an
order less or equal to $k$ which implies that $y(s)-q_2$ has order not greater than $k$. Hence $k$ is maximal and coincides with $m$.

\end{proof}

\bl \label{l:hh2} Let ${q_1}$ and ${q_2}$ be conjugate along a minimizing geodesic $\gamma(t)=\EXP_{{q_1}}(2t\lambda)$ for $t\in[0,1]$. For any $\xi \in \ker D_{2\lam}\Exp_{{q_1}}$ exists a curve 
$s\mapsto \lambda(s)$ with $\lambda(0)=\lambda$, 
$\dot\lambda(0)=\xi$ and such that if $y(s)=\Exp_{q_1}(2\lambda(s))$
then in coordinates $q_2-y(s)=O(s^3)$.
\el

\begin{proof} The proof is similar to that of Lemma \ref{l:hh}. 
Assume that $\dim \ker D_{2\lam}\EXP_{q_1}=k$, and let $z_{0}$ be the midpoint of $\g$.  
Then also $\dim \ker \Hess_{z_{0}} h_{{{q_1},{q_2}}}=k$ and by Lemma \ref{l:splitting} 
%
there exists a system of coordinates 
$(x_1,\dots,x_{n-k},\dots,x_n)$ around $z_{0}$ such that 
$h_{{{q_1},{q_2}}}(x_1,\dots,x_n)=h(z_0)+x_1^2+\dots+x_{n-k}^2+
\varphi(x_{n-k+1},\dots,x_n)$ where $\varphi$ is a smooth function. Moreover 
$\ker \Hess_{z_{0}} h_{{{q_1},{q_2}}}=\text{span}\{\partial_{x_{n-k+1}},\dots,\partial_{x_n}\}$.
 This implies in particular that the Taylor series of $\varphi$ of order 2 is zero. In addition,
since  $z_{0}$ is a minimum of $h_{q_1,q_2}$, 
the Taylor series of $\varphi$ of order 3 is also zero at this point. (See for instance the proof of Lemma \ref{l:prezele}.)
As a consequence, the differential $d\varphi$ has Taylor series of order 2 that vanishes at the origin.

Let now $w=(w_{1},\ldots,w_{n})\in \ker \Hess_{z_{0}} h_{{{q_1},{q_2}}}$.
It satisfies $w_1=\dots=w_{n-k}=0$. Let $z(\cdot)$ be defined by 
$z(s)=ws$ and $\lambda(\cdot)$ be the corresponding curve of covectors in $T_{q_1}^*M$. If we define
 $y(s)=\Exp_{q_1}(2\lambda(s))$, 
then $y(s)-q_2$  has vanishing Taylor series of order 2 (indeed it has the same order as 
$d_{z(s)}h_{{{q_1},{q_2}}}=d\varphi|_{(x_{n-k+1}(s),\dots,x_n(s))}$). Now, since the set
of vectors in the kernel of the Hessian of $h_{{q_1,q_2}}$ is
in bijection with the set of vectors in the kernel of $D_{2\lam}\Exp_{q_1}$,
the lemma is proved. 
\end{proof}


\section{Proof of Theorem \ref{t:tgv0}} \label{s-teorema-minimalita}

The idea of the proof is to identify, among the singularities that appear in the Arnol'd classification, only those that are reached by optimal geodesics. To this end, we will exploit the necessary conditions given by Lemmas \ref{l:hh} and \ref{l:hh2}. 


\bl \label{l:am}
Singularities $A_m$ are of type $(1,m)$, in the sense of Definition \ref{d:sing}.
\el
\begin{proof}
From the normal form of $A_{m}$, that can be rewritten as 
$$\Psi:(x,y_1,...,y_{m-2})\mapsto \left(x^m+\sum_{i=1}^{m-2} x^i y_i,y_1,...,y_{m-2}\right),$$ 
it is clear that $A_{m}$ is a singularity of corank 1, that is, the dimension of the kernel of the differential at zero is 1 and is spanned by $\partial_{x}$.
Moreover, the image of the curve $\gamma(s)=(s,0,...,0)$ under $\Psi$ is $\Psi(\gamma(s))=(s^m,...,0)$. Hence there exists a non-singular curve starting from the origin such that its image has order $m$ in the
variable $s$. Let us show that there exists no curve $\gamma(s)$ such that $\Psi(\gamma(s))$ has order bigger than $m$.
We can assume that $\dot{\gamma}(0)=\partial_{x}$ (otherwise $\dot\gamma(0)$ is not in the kernel of $D_{0}\Psi$ and the image $\Psi(\gamma(t))$ has order 1). It follows that $x(s)\sim s$. Assume now that $\Psi(\gamma(s))$ has order bigger than or equal to $m$. Then $y_{i}(s)$ has order bigger or equal to $m$ for every $i$. Thus the first coordinate of the image of $\Psi(\gamma(s))$, namely
$x(s)^{m}+\sum_i x(s)^{i} y_i(s)$, has order exactly $m$. Hence there is no non-singular curve passing through the origin, the image of which has a degree higher than $m$.
\end{proof}

From  Lemma \ref{l:am} and Lemma \ref{l:hh}, it follows immediately that $A_{m}$ can be optimal if and only if $m$ is odd.
Let us show now that all other singularities of the Arnol'd list can not appear as optimal singularities. Indeed it is enough to show that these singularities do not satisfy the conclusion of Lemma \ref{l:hh2}. We give a proof for the $D_{4}^{+}$ and $E_{6}^{-}$ cases. The other cases can be treated in a similar way.
%
%
%
%
\bi
\iii[$(D_{4}^{+})$]
Let us consider the singularity $D_{4}^{+}$ which has normal form 
$$\Psi:(x,y,z)\mapsto(x^2+y^2+xz,xy,z).$$ 
The kernel of the differential of $\Psi$ at zero is given by $\text{span}\{\partial_{x},\partial_{y}\}$. In particular, let us show that there exists no curve $\gamma(t)$ such that $\dot{\gamma}(0)=\partial_{x}$ and such that $\Psi(\gamma(t))$ has order greater than or equal to 3. Indeed we can assume $\gamma(t)=(at+O(t^{2}),O(t^{2}),O(t^{2}))$ where $a\neq 0$. Thus $\Psi(\gamma(t))=(a^{2}t^{2}+O(t^{3}),O(t^{3}),O(t^{2}))$ which has always order 2.
\iii[$(E_{6}^{-})$] 
The singularity $E_{6}^{-}$ has normal form
$$\Psi: (x,y,z,t,u)\mapsto(x^2+xyz+ty+ux,-y^3+x^2z+tx,z,t,u).$$ 
The kernel of the differential of $\Psi$ at zero is given by $\text{span}\{\partial_{x},\partial_{y}\}$. Let us consider a curve $\gamma(t)$ such that $\dot{\gamma}(0)=\partial_{x}$. Thus we can assume, as before, that $\gamma(t)=(at+O(t^{2}),O(t^{2}),\ldots,O(t^{2}))$, with $a\neq0$. Again, it is easy to see that $\Psi(\gamma(t))=(a^{2}t^{2}+O(t^{3}),O(t^{2}),\ldots,O(t^{2}))$ which is of order 2. 
\ei

\section{Proof of Theorem \ref{t-ogrande} and of Corollaries
\ref{c:tgv1}, \ref{c:sarrusofono1}, \ref{c:sarrusofono2}}

Suppose that $D$ is a compact set of $\bR^n$ having the origin in its
interior and that $f$ is smooth in a neighborhood of $D$. With
$x=(x_1,\ldots,x_n)$, let
\[
g(x) = g(0) + \sum_{i=1}^n x_i^{2m_i},
\]
for some integers $1\leq m_1\leq m_2 \leq \cdots \leq m_n$. Further,
we let $\ell\in\{1,\ldots, n\}$ be the smallest integer such that
$m_{\ell}=m_n$.
For functions $g$ of this form, arbitrarily many terms of the 
asymptotics of the corresponding Laplace integral are known
(as developed in the book of Kanwal and Estrada \cite{KanwalEstrada}, for example).
For our purposes, we will need only the first two terms.
In particular, we have the following asymptotic expansion as $t\searrow 0$,
\begin{equation}\begin{split}\label{Eqn:DiagonalG}
& \int_D f(x) e^{-g(x)/t} \, dx_1\cdots dx_n = \exp\lb -g(0)/t\rb
t^{\frac{1}{2m_1}+\cdots+\frac{1}{2m_n}} \times \\
&\quad  \lc \lb\prod_{i=1}^n \frac{\Gamma\lp 1/2m_i\rp}{m_i}\rb f(0) +
t^{1/m_n} \lb \frac{\Gamma(3/2m_n)}{2} \prod_{i=1}^{n-1}
\frac{\Gamma\lp 1/2m_i\rp}{m_i}\rb
\lb \sum_{\xi=\ell}^n \frac{\partial^2 f}{\partial x_{\xi}^2}(0)\rb
+ o\lp t^{1/m_n} \rp
\rc .
\end{split}\end{equation}
Here $\Gamma(\cdot)$ is the usual Gamma function; recall that $\Gamma(1/2)=\sqrt{\pi}$. This expansion is basic in what
follows.

\begin{proof}[Proof of Theorem \ref{t-ogrande}]
As above, let $\Gamma\subset M$ be the set of all midpoints
of minimizing geodesics from $q_1$ to $q_2$, and let $N(\Gamma)$ be a
sufficiently small neighborhood of $\Gamma$ (as in \cite{srneel},  where we indicate the
conditions it must be small enough to satisfy in a moment). Now by the
assumption that every minimizing geodesic from $q_1$ to $q_2$ corresponds
to a singular point of type $(1,m)$ and Lemma \ref{l:hh}, we know that
$h=h_{q_1,q_2}$ has normal form
 \[
h(x) =
d^2(q_1,q_2)/4+x_1^2+\cdots+x_{n-1}^2+x_n^{m+1}, \quad\text{for some
$m\in\{1,3,5,\ldots\}$},
\]
at each point of $\Gamma$. Hence the points of $\Gamma$ are isolated,
and since $\Gamma$ is compact, there are only finitely many points in
$\Gamma$. We enumerate them as $z_1,\ldots, z_k$.

Now for each $z_i$, let $N_i$ be a neighborhood of $z_i$ that is small
enough for there to exist coordinates on $N_i$ realizing
the normal form of $h$ (at $z_i$), small enough that they are
disjoint,  and small enough so that \cite[Theorem 27]{srneel} applies to
$N(\Gamma)=\cup_{i=1}^k N_i$. Then applying this theorem, we get the small-time asymptotic expansion, for some $\eps>0$,
\begin{equation}\label{eq:neww}
p_t(q_1,q_2) = \sum_{i=1}^k \int_{N_i} \lp\frac{2}{t}\rp^{n} e^{-h(z)/t}\lp
c_0(q_1,z)c_0(z,q_2)+O(t)\rp \, \mu(dz) +O\lp e^{-\frac{d^2(q_1,q_2)+\eps}{4t}}\rp.
\end{equation}
Here the $c_0(q_1,\cdot)$ and $c_0(\cdot,q_2)$ are the leading
coefficient functions in the Ben Arous expansion; for our purposes, it
is enough to note that they are smooth, positive functions on
$N(\Gamma)$. Further, $\mu$ is the volume measure on $M$, and we
recall that the ``$O(t)$'' term in each integral is uniform over
$N_i$. We now treat each term in this sum.

Let $m(z_i)$ be such that the minimizing geodesic through $z_i$
corresponds to a singularity of type $(1,m(z_i))$; we choose this
notation to avoid too much confusion with the notation used in
Equation \eqref{Eqn:DiagonalG}. (Also, it is again clear that we must
have $m(z_i)\in\{1,3,5,\ldots\}$.) Thus we have coordinates
$(x_1,\ldots,x_n)$ on $N_i$ such that
\[
h(x_1,\ldots,x_n) = d^2(q_1,q_2)/4+x_1^2+\cdots+x_{n-1}^2+x_n^{m(z_i)+1} .
\]
Also, because $\mu$ is a smooth measure, we have that
$\mu(dz)=F_i(x)dx_{1}\cdots dx_{n}$ for some smooth, positive function
$F_i$ (on $N_i$). Writing the integral over $N_i$ in these coordinates
and applying Equation \eqref{Eqn:DiagonalG}
(here $m_1=\cdots=m_{n-1}=1$ and $m_n=(m(z_i)+1)/2$), we have (keeping
only the leading term of the expansion)
\[\begin{split}
& \int_{N_i} \lp\frac{2}{t}\rp^{n} e^{-h(z)/t}\lp
c_0(q_1,z)c_0(z,q_2)+O(t)\rp \, \mu(dz) =  \lp\frac{2}{t}\rp^{n}
\exp\lb\frac{-d^2(q_1,q_2)}{4t}\rb \times \\
& \quad t^{\frac{n-1}{2}+\frac{1}{m(z_i)+1}} \lc \lb F_i(z_i)\lp
c_0(q_1,z_i)c_0(z_i,q_2)+O(t) \rp\rb  \Gamma\lp
\frac{1}{2}\rp^{n-1}\frac{\Gamma\lp\frac{1}{m(z_i)+1}\rp}{\frac{m(z_i)+1}{2}}
+O\lp t^{\frac{2}{m(z_i)+1}} \rp\rc.
\end{split}\]
Here we have used the uniformity of the ``$O(t)$'' to control its
integral. 
Of course, we easily simplify the above to find
\begin{equation}\label{Eqn:A1Exp}\begin{split}
& \int_{N_i} \lp\frac{2}{t}\rp^{n} e^{-h(z)/t}\lp
c_0(x,z)c_0(z,y)+O(t)\rp \, \mu(dz) = \\
&\quad \exp\lb\frac{-d^2(q_1,q_2)}{4t}\rb
\frac{1}{t^{\frac{n+1}{2}-\frac{1}{m(z_i)+1}}}
\lb C_i + O\lp t^{\frac{2}{m(z_i)+1}} \rp\rb \\
&\text{where}\quad
C_i = F_i(z_i)c_0(q_1,z_i)c_0(z_i,q_2)\frac{4}{m(z_i)+1}\lp
4\pi\rp^{\frac{n-1}{2}}
\Gamma\lp\frac{1}{m(z_i)+1}\rp .
\end{split}\end{equation}
Here the positivity of $C_i$ follows from the positivity of $F$ and
the $c_0$ on $N_i$ and the positivity of the $\Gamma$-function on the
positive reals.

Returning to Equation \eqref{eq:neww}, we see that $p_t(q_{1},q_{2})$ has
asymptotic expansion given by a  sum of $k$ terms of the form given in
Equation \eqref{Eqn:A1Exp}, for various values of $m(z_i)$. It is
clear that the resulting sum is asymptotically dominated by those
terms with the largest value of $m(z_i)$. Indeed, in the notation of
the theorem, $\ell$ is exactly this largest value. Thus, we find that
\[
p_t(q_1,q_2) = \frac{C+O\lp
t^{\frac{2}{\ell+1}}\rp}{t^{\frac{n+1}{2}-\frac{1}{\ell+1}}}e^{-d^2(q_1,q_2)/4t},
\]
where $C>0$ is the sum of the $C_i$ corresponding to geodesics (where
the geodesics are indexed by the $z_i$) which give
singularities of type $(1,\ell)$ (and by assumption, there is at least
one such geodesic). 
\end{proof}

\begin{proof}[Proof of Corollary \ref{c:tgv1}]
By Theorem \ref{t:tgv0}, every minimizing geodesic from $q_1$ to $q_2$ is
either non-conjugate, which corresponds to a non-singular point, $A_3$-conjugate, which corresponds to a
singularity of type $(1,3)$ by Lemma \ref{l:am}, or $A_5$-conjugate, which
corresponds to a singularity of type $(1,5)$, again by Lemma
\ref{l:am}.  Now case (i) of the corollary follows by \cite[Corollary 3]{srneel}, case (ii) follows by applying Theorem
\ref{t-ogrande} with $\ell=3$, and case (iii) follows by applying
Theorem \ref{t-ogrande} with $\ell=5$. 
\end{proof}

\begin{proof}[Proof of Corollary \ref{c:sarrusofono1}]
By the results on the cut locus of a generic 3D contact sub-Riemannian structure 
near the diagonal (see Section \ref{s:srcase}), for every $q_2$
sufficiently close to $q_1$, every minimizing geodesic from $q_1$ to
$q_2$ is either non-conjugate (corresponding to a non-singular point) or $A_3$-conjugate (corresponding to a
singularity of type $(1,3)$ by Lemma \ref{l:am}).  
Thus case (i) of the corollary follows
by \cite[Corollary 3]{srneel}. Also,
case (ii) of the corollary follows by applying Theorem \ref{t-ogrande}
with $n=3$ and $\ell=3$ and from the earlier observation that
$A_3$-conjugate geodesics arise arbitrarily close to $q_1$. 
\end{proof}

\begin{proof}[Proof of Corollary \ref{c:sarrusofono2}]
The proof is almost identical to that of Corollary
\ref{c:sarrusofono1}. In particular, we again already know that, for
$q_1$ as in the corollary and $q_2$ sufficiently close to $q_1$, every
minimizing geodesic from $q_1$ to $q_2$ is either non-singular or $A_3$-conjugate (and of type $(1,3)$), and that there
are points arbitrarily close to $q_1$ for which the second possibility
occurs. Thus case (i) of the corollary
follows  by applying \cite[Corollary 3]{srneel},
and case (ii) of the corollary follows by applying Theorem
\ref{t-ogrande} with $n=4$ and $\ell=3$. \end{proof}


\section{Proof of the complementary results}

\subsection{Proof of Theorem \ref{t:zele}}

\begin{lemma}\label{l:prezele}
Let $f$ be a smooth function defined on a neighborhood of the origin in
$\bR^n$. Suppose that $f$ has an isolated local minimum at the origin,
and that the Hessian of $f$ at the origin has rank $n-r$ for some
$r\in \{0,1,\ldots,n-1\}$. Then there exist a system of coordinates
$x=(x_1,\ldots,x_n)$ around the origin such that, in some neighborhood
of the origin,
\[
x_1^2+\cdots +x_{n-r}^2  \leq f(x)-f(0) \leq x_1^2+\cdots +x_{n-r}^2
+x_{n-r+1}^4 +\cdots + x_n^4 .
\]
\end{lemma}

\begin{proof}
It's clear that $f-f(0)$ satisfies the assumptions of Lemma
\ref{l:splitting}, and thus there are coordinates $(u_1,\ldots,u_n)$
near the origin and a smooth function $\phi:\bR^k\rightarrow \bR$ such
that, in some neighborhood of the origin
\begin{equation}\label{Eqn:fsplit}
f(u_1,\ldots,u_n)-f(0) = u_1^2+\cdots +u_{n-r}^2 +\phi(u_{n-r+1},\ldots,u_n) .
\end{equation}
Further, $\phi$ must have an isolated local minimum of 0 in some
neighborhood of the origin of $\bR^k$.

Next, we wish to show that, for some $C>0$,
$\phi(u_{n-r+1},\ldots,u_n)<C(u_{n-r+1}^4 +\cdots + u_n^4)$ (assuming
$r\geq1$, otherwise we consider this to be vacuously true). To
simplify notation, let $y_1= u_{n-r+1}, \ldots, y_r= u_n$. We know
that $\phi$ is smooth, has a local minimum of 0 at the origin, and its
Hessian at the origin vanishes. Thus, if we write $\phi$ as a
third-order Taylor series (with remainder) around the origin, the
linear and quadratic terms all vanish. If we let $y=(y_1,\ldots,y_r)$, we have
$\phi(y)=P_3(y)+P_4(y)+R_5(y)$ where $P_3$ and $P_4$ are homogeneous polynomials of degree 3 and 4, respectively, and $R_5$ is a smooth function such that $R_5(y)=o(\|y\|^4)$.

For any monomial of degree four $y_iy_jy_ky_l$, it is clear that $y_iy_jy_ky_l\leq \sum_{i=1}^r y_i^4$. Hence there exists a constant $C_4>0$ such that $|P_4(y)|\leq C_4\sum_i y_i^4$. Moreover, since $\|y\|^4=(\sum_i y_i^2)^2$ it is clear that there exists a constant $C_5>0$ such that $\|y\|^4\leq C_5\sum_i y_i^4$ and hence $R_5(y)\leq C_5\sum_i y_i^4$ in a neighborhood of the origin. To complete the argument, we wish to prove that $P_3=0$. Assume by contradiction that there exists $y$ such that $P_3(y)\neq 0$. Hence for $t\in\R$,
$\phi(ty)=t^3P_3(y)+o(t^3)$. If $P_3(y)>0$ (resp.\ $<0$) then, for $t$
small enough and negative (resp.\ positive), $\varphi(ty)<0$ which contradicts  the hypotheses.

We conclude that $\phi(y)=P_4(y)+R_5(y)$, and the previous considerations imply that there exists $C$ such that, in a neighborhood of the origin, $0\leq \varphi(y)\leq C \sum_{i=1}^r y_i^4$, as desired. Thus, near the origin,
\[
u_1^2+\cdots +u_{n-r}^2  \leq f(u)-f(0) \leq u_1^2+\cdots +u_{n-r}^2 +
C\lp u_{n-r+1}^4 +\cdots + u_n^4 \rp.
\]
Letting each $x_i$ be an appropriate rescaling of $u_i$ then gives a system of
coordinates as claimed in the lemma. 
\end{proof}

\begin{proof}[Proof of Theorem \ref{t:zele}]
Because there is a unique minimizing geodesic from $q_1$ to $q_2$, we
know that $h=h_{q_1,q_2}$ has a unique minimum at the midpoint of this
geodesic. Further, from Theorem \ref{t:handconj}, we see that the
Hessian of $h$ at this midpoint has rank $n-r$. If, as usual, we let
$N(\Gamma)$ be a (sufficiently small) neighborhood of this midpoint,
then Lemma \ref{l:prezele} applies to $h-\frac{1}{4}d^2(q_1,q_2)$, and
we conclude that there are coordinates on $N(\Gamma)$ such that
\[
x_1^2+\cdots +x_{n-r}^2  \leq h(x)-\frac{1}{4}d^2(q_1,q_2) \leq
x_1^2+\cdots +x_{n-r}^2 +x_{n-r+1}^4 +\cdots + x_n^4 .
\]

Now we can apply Theorem 27 of \cite{srneel} to get the small-time
asymptotic expansion
\[
p_t(q_1,q_2) = \int_{N(\Gamma)} \lp\frac{2}{t}\rp^{n} e^{-h(z)/t}\lp
c_0(q_1,z)c_0(z,q_2)+O(t)\rp \, \mu(dz) .
\]
Here the $c_0(q_1,\cdot)$ and $c_0(\cdot,q_2)$ are smooth, positive
functions on $N(\Gamma)$, $\mu$ is the volume measure on $M$, and the
``$O(t)$'' term is uniform. Estimating $h$ from both sides in the
$x$-coordinates gives
\[\begin{split}
p_t(q_1,q_2) &\leq
\int_{N(\Gamma)} \lp\frac{2}{t}\rp^{n} \exp\lb-\frac{h(0)+x_1^2+\cdots
+x_{n-r}^2}{t} \rb \lp
c_0(q_1,x)c_0(x,q_2)+O(t)\rp \, \mu(dx) \\
\text{and}\quad
p_t(q_1,q_2) &\geq
\int_{N(\Gamma)} \lp\frac{2}{t}\rp^{n} \exp\lb-\frac{h(0)+x_1^2+\cdots
+x_{n-r}^2+x_{n-r+1}^4 +\cdots + x_n^4}{t} \rb \\
&\qquad \times \lp
c_0(q_1,x)c_0(x,q_2)+O(t)\rp \, \mu(dx) .
\end{split}\]
Let $F(x)$ be the density of $\mu$ with respect to $dx_1\cdots dx_n$;
then $F$ is smooth and positive.

The idea is to apply Equation \eqref{Eqn:DiagonalG} to the right-hand
side of each of the above in order to determine the leading term in
the asymptotic expansions. Starting with the first equation, note that
we are free to make $N(\Gamma)$ smaller without affecting the
asymptotics. Thus, assume that it is of the form $x_i\in (-\eps,\eps)$
for all $i$. Then we can use Fubini's theorem to write
\[\begin{split}
& \int_{N(\Gamma)} \lp\frac{2}{t}\rp^{n}
\exp\lb-\frac{h(0)+x_1^2+\cdots +x_{n-r}^2}{t} \rb \lp
c_0(q_1,x)c_0(x,q_2)+O(t)\rp \, \mu(dx) \\
&= \lp\frac{2}{t}\rp^{n}e^{-d^2(q_1,q_2)/4t}
\int_{(-\eps,\eps)^{n-r}} \exp\lb-\frac{x_1^2+\cdots +x_{n-r}^2}{t}
\rb \times \\
& \quad \lc \int_{(-\eps,\eps)^{m}} \lp
c_0(q_1,x)c_0(x,q_2)+O(t)\rp F(x) \, dx_{n-r+1}\cdots dx_n\rc \,
dx_1\cdots dx_{n-r}
\end{split}\]
We see that the inner integral is just some smooth, positive function,
call it $g(x_1,\ldots,x_{n-r})$, plus a uniform $O(t)$. Then applying
Equation \eqref{Eqn:DiagonalG} to the outer Laplace integral, we find
\[\begin{split}
p_t(q_1,q_2) &\leq  \lp\frac{2}{t}\rp^{n}e^{-d^2(q_1,q_2)/4t}
t^{\frac{n-r}{2}}\Gamma\lp \frac{1}{2}\rp^{n-r}\lb g(0)+O(t)\rb \\
&= \frac{C+O(t)}{t^{\frac{n}{2}+\frac{r}{2}}} e^{-d^2(q_1,q_2)/4t}  ,
\end{split}\]
for some $C>0$.

The second equation is a bit simpler, since we don't need to split the
integral. We have
\[\begin{split}
&p_t(q_1,q_2)\geq \int_{N(\Gamma)} \lp\frac{2}{t}\rp^{n}
\exp\lb-\frac{h(0)+x_1^2+\cdots +x_{n-r}^2+x_{n-r+1}^4 +\cdots +
x_n^4}{t} \rb \\
&\qquad\qquad\qquad\qquad \times \lp
c_0(q_1,x)c_0(x,q_2)+O(t)\rp \, \mu(dx) \\
& \quad=  \lp\frac{2}{t}\rp^{n}  e^{-d^2(q_1,q_2)/4t}
t^{\frac{n-r}{2}+\frac{r}{4}}\lb \Gamma\lp \frac{1}{2}\rp^{n-r}
\frac{\Gamma\lp \frac{1}{4}\rp^{r}}{2^r}  c_0(q_1,0)c_0(0,q_2)F(0)
+O\lp t^{1/2}\rp \rb \\
&\quad = \frac{\tilde{C}+O\lp t^{1/2}\rp}{t^{\frac{n}{2}+\frac{r}{4}}}
e^{-d^2(q_1,q_2)/4t} ,
\end{split}\]
for some $\tilde{C}>0$.

Combining these two estimates on $p_t(q_1,q_2)$ and using the
definition of ``big-$O$'' shows that we can find constant
$C_1,C_2,t_0>0$ such that the conclusion of the theorem holds. 
\end{proof}

\subsection{Proof of Theorem \ref{t:examples}}

The proof consists of the construction of the family of metrics the existence of which is claimed in the theorem. The essential idea is that the local structure of $h_{q_{1},q_{2}}$ is determined by the order of contact of the spheres of radius $d(q_{1},q_{2})/2$ centered at $q_{1}$ and $q_{2}$. (These spheres are tangent exactly at points in $\Gamma$.) 

 In $\bR^2$, consider a neighborhood $U$ of the segment
of the $y$-axis from $(0,-1/2)$ to $(0,1/2)$ (we will eventually take
these points to be $q_1$ and $q_2$). We wish to put a metric on this neighborhood such that
the resulting heat kernel for these two points has the desired
small-time expansion. The proof is somewhat involved, so we divide it into four steps.

\emph{Step 1}: In this step, our task is to indicate the local
structure of the geodesics on this neighborhood.

Let $\xi$ be the piece of the circle of radius $1/2$ around $(0,1/2)$
contained in $U$ and passing through $(0,0)$. That is, $\xi$ is just
the graph of
\[
y(x)= \frac{1}{2}-\sqrt{\frac{1}{4}-x^2} = x^2+x^4+2x^6+5x^8 +\cdots
\]
in some neighborhood of the origin (note that every term $x^{2k}$
appears with non-zero coefficient). Next, for any
integer $\eta$, with $\eta\geq 3$, let $\gamma_{\eta}$ be the curve
given by the graph, near the origin, of the first $\eta-1$ terms in
this power series; for example,
\[\begin{split}
\gamma_3 = \lc (x,y): y= x^2+x^4 \text{ for $x\in(-\eps,\eps)$} \rc 
\end{split}\]
The important point is that $\xi$ and $\gamma_{\eta}$ are tangent to
order $2\eta-1$ but not to order $2\eta$ at the origin.

Next, we describe the desired geodesics starting from $(0,-1/2)$. Let
$f_{\theta}(s)$, for
$\theta\in [0,2\pi)$ and $s\in [0,1]$ be a family of curves (indexed
by $\theta$) with the following properties. We have that
$f_{\theta}(0)=(0,-1/2)$ for all
$\theta$, and in a neighborhood of $(0,-1/2)$, the curves (for small
$s$) are the usual radial lines coming from $(0,-1/2)$, parametrized
by $\theta$ as
usual, except that for convenience we let $\theta=0$ correspond to the
ray upward along the $y$-axis and we let $\theta$ increase in the
clockwise
direction. Further, we let $f_0(s)$ (for $s\in[0,1]$, of course) be
the segment of the $y$-axis from $(0,-1/2)$ to $(0,0)$. To extend the
rest of the curves,
we first insist that $f$, as a function of $\theta$ and $s$, gives a
smooth diffeomorphism onto the portion of $U$ below (and including)
$\gamma_{\eta}$,
except of course for $s=0$. For simplicity, we assume that $f$ is
symmetric with respect to reflection across the $y$-axis, that is,
taking $\theta$ to
$-\theta$ gives a curve which is the reflection of $f_{\theta}$ across
the $y$-axis. Next, for all $\theta$ near $0$, the curve $f_{\theta}$
should intersect
$\gamma_{\eta}$ perpendicularly (the point of intersection occurring
when $s=1$), and further, for $s$ near 1, $f_{\theta}(s)$ should be
straight (with respect
to the Euclidean metric, so that near $\gamma_{\eta}$, $f_{\theta}$
corresponds with the normal lines to $\gamma_{\eta}$).

For future use, let $\sigma$ be the arclength parameter on
$\gamma_{\eta}$, normalized so that $\gamma_{\eta}(0)=(0,0)$ and the
$x$-coordinate increases as
$\sigma$ increases. Since $f_{\theta}(1)$ for $\theta$ near 0 also
provides a parametrization of $\gamma_{\eta}$, we let $\sigma(\theta)$
denote the induced
change of coordinates. Our normalizations are such that $\sigma(0)=0$.
Further, by the assumptions that $f$ gives a diffeomorphism and is
symmetric (as
discussed above), we see that $\sigma^{\prime}(0)>0$ and
$\sigma^{\prime\prime}(0)=0$. Finally, note that we can prescribe any
positive value for
$\sigma^{\prime}(0)$ and any real value for
$\sigma^{\prime\prime\prime}(0)$, in the sense that we can find a
family of curves $f$ that realize these
values.

\emph{Step 2}: As mentioned, the curves $f_{\theta}(s)$ are intended
to be the geodesic rays starting from $(0,-1/2)$, relative to some
metric on $U$. In this step, we describe this metric.

In particular, we take a smooth metric on $U$ with the following properties. The
geodesic rays from $(0,-1/2)$ coincide with the curves $f_{\theta}$
until these curves
hit either the boundary of $U$ or $\gamma_{\eta}$. Near
$\gamma_{\eta}$, these geodesic rays are straight lines (hence
Euclidean geodesics as well). The
metric is Euclidean above and in a neighborhood of $\gamma_{\eta}$, and thus these geodesics continue as
straight lines, normal to $\gamma_{\eta}$, on the part of $U$ above
$\gamma_{\eta}$. If $r$ is the distance function from
$(0,-1/2)$, $(r,\theta)$ for $\theta$ as above gives
polar coordinates around
$(0,-1/2)$ (with the above convention for $\theta$), and $\gamma_{\eta}$ coincides, near
$(0,0)$, with the locus of
points that are distance $1/2$ from $(0,-1/2)$. Finally, this metric on $U$ is symmetric with respect to reflection about
the $y$-axis. (That such a metric exists is a straightforward exercise, and we now assume that $U$ is equipped with some such metric.)

Also,
$\xi$ coincides, near $(0,0)$, with the locus of points that are
distance $1/2$ from
$(0,1/2)$. Since $\xi$ and $\gamma_{\eta}$ meet only at $(0,0)$, we see that the
unique minimizing geodesic from $(0,-1/2)$ to $(0,1/2)$ is the
intervening segment of the
$y$-axis, and these points are distance 1 from each other.

Finally, we include $U$ isometrically in the smooth 2-sphere (with some smooth metric), in such  a way that  the only minimizing geodesic from $(0,-1/2)$ to $(0,1/2)$ remains this
segment of the $y$-axis. Having done this, it is clear that the small-time
asymptotics of the heat kernel relative to $(0,-1/2)$ and $(0,1/2)$
depend only on the metric in a neighborhood of this segment of the $y$-axis, so we can again restrict
our attention to $U$. We will take $q_1$ to be $(0,-1/2)$ and $q_2$ to
be $(0,1/2)$.

\emph{Step 3}:
Having constructed our metric (or more accurately, family of metrics),
the next step is to understand $h=h_{q_1,q_2}$ near $(0,0)$. We know
that $h$ has a
unique minimum at $(0,0)$, where it equals $1/4$. We know that the
Hessian of $h$, restricted to the direction along the $y$-axis, is
non-degenerate, so
Lemma \ref{l:splitting} implies that there are coordinates $(u_1,u_2)$
around $(0,0)$ such that $h=\frac{1}{4}+u_1^2+\phi(u_2)$, for some
smooth function
$\phi$. Next, note that, for $(x,y)$ near $(0,0)$, the distance from
$(0,1/2)$ to $(x,y)$ can be determined solely from the distance to
$\xi$, in the
Euclidean metric (and whether $(x,y)$ is above or below $\xi$), and
similarly for $(0,-1/2)$ and $\gamma_{\eta}$. In particular, near
$(0,0)$, $h(x,y)$ is
completely determined by $\xi$ and $\gamma_{\eta}$ (and the fact that
our metric is Euclidean near $(0,0)$). Thus, $h$ depends on ${\eta}$,
but not on
$\sigma(\theta)$. Using this relatively simple structure for $h$ and
the fact that $\xi$ and $\gamma_{\eta}$ are tangent to order $2\eta-1$
but not to order $2\eta$
at the origin, we see that there is some constant $C>1$ such that, if
$(x,y)$ is near $(0,0)$ and between $\xi$ and $\gamma_{\eta}$
\[
\frac{1}{C}x^{2\eta} < h(x,y)-\frac{1}{4} < C x^{2\eta} .
\]
It follows that, after possibly re-parametrizing $u_2$, we
have\footnote{In fact, it's not hard to show that the curve
$u_1=0$ is exactly the curve of equidistant points, which is squeezed
between $\xi$ and $\gamma_{\eta}$, and thus give a more concrete
description of the
$(u_1,u_2)$ coordinates. But this isn't necessary and would only make
this long construction even longer.}
$h=\frac{1}{4}+u_1^2+u_2^{2\eta}$.

\emph{Step 4}: In this step, we determine the asymptotic expansion,
and conclude the proof.

We have just seen that the unique minimizing geodesic from $q_1=(0,-1/2)$
to $q_2=(0,1/2)$ has a singularity of type $(1,2\eta-1)$. Thus,
Theorem \ref{t-ogrande} applies, but we want a more detailed
description of the asymptotic expansion. We apply Equation
\eqref{Eqn:DiagonalG}, as in the proof of Theorem \ref{t-ogrande}, except
that now we keep track of the coefficients of the first two terms of
the expansion. From this, we have
\begin{equation}\label{Eqn:ExampleExpansion}\begin{split}
p_t (q_1,q_2)&= e^{-1/4t}\frac{1}{t^{(3\eta-1)/2\eta}} \lc
KF(0,0)c_0((0,-1/2),(0,0))c_0((0,0),(0,1/2)) \right. \\
& \quad \left. + t^{1/\eta}\tilde{K}
\frac{\partial^2}{\partial u_2^2}\lb
F(\cdot)c_0((0,-1/2),\cdot)c_0(\cdot,(0,1/2)) \rb_{(0,0)}+o\lp
t^{1/\eta}\rp\rc ,
\end{split}\end{equation}
where $K$ and $\tilde{K}$ are some positive constants depending only
on $\eta$. Here we recall that $F(u_1,u_2)$ is the density of Lebesgue
measure $dxdy$
with respect to $du_1du_2$, and that the $c_0$ are the coefficients of
the leading term in the Minakshishundaram-Pleijel expansion. All three
are smooth,
positive functions. As our discussion of the coordinates $(u_1,u_2)$
makes clear, $F$ depends only on $\eta$ (via $\gamma_{\eta}$), and not
on $\sigma(\theta)$.
This property is sufficient for us here, so we make no effort to
compute it more explicitly. Next, $c_0$ is given by the  reciprocal
of the square-root of
the Jacobian of the exponential map (see, for example, Section 5.1 of
Hsu's book \cite{EltonBook} for further discussion). Since the
metric is just the Euclidean metric near the minimizing geodesic from
$(0,1/2)$ to $(0,0)$ (recall that this geodesic is just a segment of
the $y$-axis), we see
that $c_0((x,y),(0,1/2)) = 1$ for all $(x,y)$ near $(0,0)$.

Finally, we need to understand $c_0((0,-1/2),(x,y))$ near the origin.
In polar coordinates around $(0,-1/2)$, the (Euclidean) volume form on
the tangent
space is $r dr d\theta$, which is just $1/2 \, dr d\theta$ when
$r=1/2$, while the volume form induced by the metric on $U$ (which is
Euclidean for
different coordinates) is, again when $r=1/2$, given by
$\sigma^{\prime}(\theta)drd\theta$. Thus, if $(x,y)$ is given in polar
coordinates by $(1/2,\theta)$
for $\theta$ near 0, we have
$c_0((0,-1/2),(x,y)) = \sqrt{1/2\sigma^{\prime}(\theta)}$.
It follows that the leading coefficient in the expansion
\eqref{Eqn:ExampleExpansion} is
\[
K \cdot F(0,0)/\sqrt{2\sigma^{\prime}(0)} .
\]
Since $K F(0,0)$ is some positive constant depending only on $\eta$
(given our general construction of the metric), we can make this be
any positive real by
prescribing the appropriate (positive) value for $\sigma^{\prime}(0)$,
and we do this so as to get the $\alpha$ we want in the expansion.
Here, of course, we recall our earlier observation that we are able to
choose
$\sigma^{\prime}(0)$ in precisely this way.

We now consider the second term in the expansion. It's clear from our
discussion of the $(u_1,u_2)$ coordinates that the curve $u_1=0$ is
squeezed
between $\xi$ and $\gamma_{\eta}$. Hence $\partial u_2$ at $(0,0)$ is
parallel to the $x$-axis. So by symmetry, the first derivatives of $F$
and the $c_0$'s all
vanish at the origin. 

Then using all of this plus the fact that $c_0((x,y),(0,1/2))$ is
locally constant, we see that
\[
\frac{\partial^2}{\partial u_2^2}\lb
F(\cdot)c_0((0,-1/2),\cdot)c_0(\cdot,(0,1/2)) \rb_{(0,0)} =  \sqrt{\frac{1}{2\sigma^{\prime}(0)}}
\frac{\partial^2}{\partial u_2^2} F(0,0)
+ F(0,0) \frac{\partial^2}{\partial u_2^2}
\sqrt{\frac{1}{2\sigma^{\prime}(0)}} .
\]
Because the curve $u_1=0$ is squeezed between $\xi$ and
$\gamma_{\eta}$ and $\xi$ and $\gamma_{\eta}$ are tangent to at least
order 5, we see that $u_2$ is a
re-parametrization of $\gamma_{\eta}$ up to more than order 2 (at the
origin), and thus it makes sense to compute
\[
\frac{\partial^2}{\partial u_2^2} \sigma^{\prime}(0) =
\sigma^{\prime\prime\prime}(0)  \lp\frac{\partial\theta}{\partial u_2}
(0)\rp^2 +
\sigma^{\prime\prime}(0) \frac{\partial^2\theta}{\partial u_2^2}(0)
= \sigma^{\prime\prime\prime}(0)  \lp\frac{\partial\theta}{\partial
u_2}(0)\rp^2 ,
\]
where we've used that $\sigma^{\prime\prime}(0)$ vanishes by symmetry.
Since $u_2(\sigma)=c \sigma +O(\sigma^2)$ for some $c>0$ depending
only on
$\eta$, we are able to see that
\[
\frac{\partial^2}{\partial u_2^2}  \sqrt{\frac{1}{2\sigma^{\prime}(0)}}
= -\frac{1}{2\sqrt{2}}
\frac{1}{\sigma^{\prime}(0)^{3/2}}\sigma^{\prime\prime\prime}(0)
\lp\frac{1}{c\sigma^{\prime}(0)}\rp^2 .
\]
It follows that the coefficient of the second term in the expansion \eqref{Eqn:ExampleExpansion} is
\bqn \nn
\tilde{K}\lc \sqrt{\frac{1}{2\sigma^{\prime}(0)}}
\frac{\partial^2F}{\partial u_2^2} (0,0) -
F(0,0) \frac{1}{2\sqrt{2}}
\frac{1}{c^2\sigma^{\prime}(0)^{7/2}}\sigma^{\prime\prime\prime}(0)\rc .
\eqn
Since we've already prescribed $\sigma^{\prime}(0)$, we see that there
is a unique (real) value for $\sigma^{\prime\prime\prime}(0)$ that
will make this
expression equal to any given value. Since we are free to prescribe
this value, independently of everything else in  the expression, we do
so in order to get the desired $\beta$ in the expansion. Thus, we have
produced a metric
with the desired properties, proving the theorem.

\paragraph{Acknowledgements.} 
This work has been supported by the European Research Council, ERC StG 2009 \virg{GeCoMethods}, contract number 239748.
{\small
\bibliographystyle{siam}
\bibliography{SRGeneric-Biblio-new.bib}
}


\medskip

\noindent
\textsc{Davide Barilari}\\
{\footnotesize Universit\'e Paris Diderot - Paris 7, Institut de Mathematique de Jussieu, UMR CNRS 7586 -
UFR de Mathématiques. \\ 
CNRS, CMAP \'Ecole Polytechnique and \'Equipe INRIA GECO Saclay \^Ile-de-France, Paris.}\\
{\footnotesize  E-mail:  
{\tt \href{mailto:barilari@math.jussieu.fr}{barilari@math.jussieu.fr}}}\\

\noindent
\textsc{Ugo Boscain}\\
{\footnotesize CNRS, CMAP \'Ecole Polytechnique and \'Equipe INRIA GECO Saclay \^Ile-de-France, Paris. \\ E-mail: {\tt \href{mailto:boscain@cmap.polytechnique.fr}{boscain@cmap.polytechnique.fr}}}\\

\noindent
\textsc{Gr\'egoire Charlot}\\
{\footnotesize Institut Fourier, UMR 5582, Universit\'e Grenoble 1 and \'Equipe INRIA GECO Saclay \^Ile-de-France, Paris.\\ E-mail: {\tt \href{mailto:charlot@ujf-grenoble.fr}{charlot@ujf-grenoble.fr}}}\\

\noindent
\textsc{Robert W. Neel}\\
{\footnotesize Department of Mathematics, Lehigh University, Bethlehem, PA, USA	\\ E-mail: {\tt \href{mailto:robert.neel@lehigh.edu}{robert.neel@lehigh.edu}}}

%
%

\end{document}